\documentclass[12pt]{article}
\usepackage{amsmath, amsthm, amssymb, amstext}
\usepackage{array, amsfonts, mathrsfs}
\usepackage{latexsym, amsmath}
\usepackage{graphicx,color}
\usepackage{epsf}

\theoremstyle{plain}

\newtheorem{Theorem}{Theorem}
\newtheorem{Lemma}{Lemma}
\newtheorem{Proposition}{Proposition}

\newtheorem*{Remark*}{Remark}

\newtheorem{Condition}{Condition}

\newtheorem*{Corollary*}{Corollary}

\begin{document}
\date{}
\author{
M. I. Belishev
\thanks {
St. Petersburg Department of V. A. Steklov Institute of Mathematics of the Russian Academy of Sciences,
27 Fontanka, St. Petersburg, 191023 Russia
belishev@pdmi.ras.ru.
The work was supported by the grant RFBR 17-01-00529A.},\,
S. A. Simonov
\thanks {
St. Petersburg Department of V. A. Steklov Institute of Mathematics of the Russian Academy of Sciences, 27 Fontanka, St. Petersburg, 191023 Russia;
St. Petersburg State University, 7--9 Universitetskaya nab., St. Petersburg, 199034 Russia,
sergey.a.simonov@gmail.com.
The work was supported by the grants RFBR 18-31-00185mol\_a, 17-01-00529A, 16-01-00443A and 16-01-00635A.}
}
\title{The wave model of metric spaces}

\maketitle
\begin{abstract}
Let $\Omega$ be a metric space, $A^t$ denote the metric neighborhood of the set $A\subset\Omega$ of the radius $t$; ${\mathfrak O}$ be the lattice of open sets in $\Omega$ with the partial order $\subseteq$ and the order convergence. The lattice of  $\mathfrak O$-valued functions of $t\in(0,\infty)$ with the point-wise partial order and convergence contains the family ${I\mathfrak O}=\{A(\cdot)\,|\,\,A(t)=A^t,\,\,A\in{\mathfrak O}\}$. Let $\widetilde\Omega$ be the set of atoms of the order closure $\overline{I\mathfrak O}$. We describe a class of spaces for which the set $\widetilde\Omega$, equipped with an appropriate metric, is isometric to the original space $\Omega$.

The space  $\widetilde\Omega$ is the key element of the construction of the {\it wave spectrum} of a symmetric operator semi-bounded from below, which was introduced in a work of one of the authors. In that work, a program of constructing a functional model of operators of the aforementioned class was devised. The present paper is a step in realization of this program.
\end{abstract}



\setcounter{section}{-1}

\section{Introduction}\label{sec Introductioneq 2 lower boundeq 2 lower boundeq 2 lower bound}
In the paper \cite{B JOT}, a program of constructing a new functional model of sym\-met\-ric semi-bounded operators was devised. This model was named {\it the wave model}, which is motivated by its origin in inverse problems of mathematical physics. In the works \cite{BSim_1, Simonov 2017} its systematic investigation has been started. The key element of the wave model is {\it the wave spectrum} --- a unitary invariant of a symmetric semi-bounded operator \cite{B JOT}. This is a topological space determined by the operator. As it turned out, the fundamental problem of reconstruction of a Riemannian manifold from the boundary spectral and dynamical data \cite{B_UMN} can be reduced to finding the wave spectrum of the minimal Laplacian corresponding to the manifold. This spectrum, equipped with an appropriate metric, firstly, is determined by the inverse problem data, and secondly, turns out to be isometric to the manifold that has to be recovered. Thus it solves the problem.

The subject of the present paper is general properties of the wave spectrum as a set-theoretic construction, and its possible structure. Here we study them separately, without connection to their origin --- semi-bounded operators. The results of the paper, in our opinion, give reason to speak of the wave spectrum as a new and rich in content attribute of this important class of operators.

Relatively simple facts are placed into Propositions, their proofs are not included. In the text some inaccuracies from \cite{B JOT} were corrected. We are grateful to the referee for helpful remarks on the text.

\section{The lattice $\mathfrak O$}
Everywhere in the paper, $(\Omega,d)$ is a complete metric space. Definitions and terminology are taken mostly from \cite{Birkhoff} and \cite{Vulih}.
\smallskip

\noindent$\bullet$\,\,\, The metric neighborhood of a set $A\subset\Omega$ of the radius $t$ is the set
\begin{equation*}\label{eq 0 metric neighborhood}
A^t:=\{x\in\Omega\,|\,\, d(x,A)<t\},\qquad t>0;
\end{equation*}
let us also denote $\emptyset^t:=\emptyset$.  By ${\rm int}A$ we denote the set of inner points of $A$; $\overline A$ is the closure of $A$ in $\Omega$. We note that $A^t=\overline A^{\,t}$ for $t>0$.

Let $B_r(x):= \{y\in \Omega\ |\ d(x,y)<r\}$ and $B_r[x]:= \{y\in\Omega\ |\ d(x,y)\leqslant r\}$ be the open and the closed balls centered at $x\in\Omega$ and with the radius $r>0$.
\smallskip

\noindent$\bullet$\,\,\, Let $\mathfrak O$ be the lattice of open sets in $\Omega$ with the partial order $\subseteq$ and the operations $G_1\vee G_2= G_1\cup G_2$ and $G_1\wedge G_2=G_1\cap G_2$. It is {\it complete}: every family of sets $G_\alpha$ in it has the least upper bound $\bigcup_\alpha G_\alpha$ and the greatest lower bound ${\rm int}\bigcap_\alpha G_\alpha$. In particular, the lattice $\mathfrak O$ has the least and the greatest elements ---
the sets $\emptyset$ and $\Omega$, respectively.

In the lattice $\mathfrak O$ one can introduce {\it the order convergence} \cite{Birkhoff}. Let $\{G_{\alpha}\}$ be a net in $\mathfrak O$; then by definition
$$
G_{\alpha}  \overset{o}\to   G \text{, if}\quad\underset{\alpha}{\rm sup}\,{\underset{\beta>\alpha}{\rm inf}G_\beta} =\underset{\alpha}{\rm inf}{\,\underset{\beta>\alpha}{\rm sup}\,G_\beta} = G\,,
$$

$$
G_{\alpha} \overset{o}\to   G \text{, if}\quad \bigcup_{\alpha}\text{int}\bigcap_{\beta>\alpha}G_{\beta}   = {\rm int}\bigcap_{\alpha}\bigcup_{\beta>\alpha}G_{\beta}   =   G.
$$
The net $\{G_\alpha\}$ grows (decreases), if $\alpha\geqslant\beta$ implies $G_\alpha\geqslant G_\beta$\, ($G_\alpha\leqslant G_\beta$). Let us note a simple fact.

\begin{Proposition}\label{prop 1 monotonic}
Every growing (decreasing) net $\{G_{\alpha}\}$ of sets from $\mathfrak O$ converges in order to its greatest lower bound ${\rm int}\bigcap\limits_{\alpha}G_{\alpha}$ (to its least upper bound $\bigcup\limits_{\alpha}G_{\alpha}$).
 \end{Proposition}

Let  ${\cal F}\subset\mathfrak O$ be a family of open subsets. Denote the set of the limits of all the order-convergent nets from ${\cal F}$ by $\overline{\cal F}$ and call the operation ${\cal F}\mapsto \overline{\cal F}$ {\it the order closure}\footnote{We note that, generally speaking, $\overline{\overline {\cal F}}\not=\overline {\cal F}$: cf. \cite{Vulih}. Nevertheless, the $o$-convergence determines a topology in $\mathfrak O$. In it, the families with the property $\mathcal F=\overline{\mathcal F}$ are closed by definition.}.

\section{The set $\overline{I\mathfrak O}$}
\noindent$\bullet$\,\,\,Consider the family ${\mathfrak F}:=F\left((0,+\infty);{\mathfrak O}\right)$ of functions on the half-line taking values in the lattice $\mathfrak O$. It is easy to see that $\mathfrak F$ is a complete lattice with respect to {\it the point-wise order}
$$
f\leqslant g  \Leftrightarrow f(t)\subseteq g(t), \quad  t>0,
$$
which determines the operations
$$
(f\vee g)(t):=f(t)\vee g(t),\quad(f\wedge g)(t):=f(t)\wedge g(t),\quad t> 0.
$$
There exist the least and the greatest elements $0_{\mathfrak F}(\cdot)\equiv\emptyset$ and $1_{\mathfrak F}(\cdot)\equiv\Omega$. The order convergence in ${\mathfrak F}$ coincides with {\it point-wise order convergence}:
$$
f_\alpha\overset{o}\to f\,\,\Leftrightarrow\,\,f_\alpha(t)\overset{o}\to f(t),\quad t> 0\,.
$$

\noindent$\bullet$\,\,\,A map between two partially ordered sets $i:{\mathscr P}\to{\mathscr Q}$ is called isotonic ({\it an isotony}), if it preserves the order, i. e., $p\leqslant q$ implies $i(p)\leqslant i(q)$\, \cite{Birkhoff}. We call {\it the metric isotony} the map $I:{\mathfrak O}\to{\mathfrak F}$, $(IG)(t):=G^t,\,\,\,t> 0$. Note that $I\emptyset=0_{\mathfrak F}$ according to the above definitions.

Consider the image of the whole lattice
\begin{equation*}\label{eq 2 IO}
I\mathfrak O   =   \{g\ |\ g(t)=G^t,\,\,G\in\mathfrak O,\,\,\,\,t>0\}\, \subset \,  \mathfrak F\,.
\end{equation*}
Its order closure $\overline{I\mathfrak O}$ in the lattice $\mathfrak F$ is of special interest to us. Note that, generally speaking, it is not a lattice. Let us list some of the properties of its elements.

\begin{Proposition}\label{prop 2 convergence 2 monotonic + convergence}
All the elements of $\overline{I\mathfrak O}$ are growing functions. For every $g   \in   \overline{I\mathfrak O}$ there exists a decreasing net $\{G_{\alpha}\}$ in $\mathfrak O$ such that $IG_{\alpha}   \overset{o}\to   g$ in $\mathfrak F$. Furthermore, $g(t)   = {\rm int}\bigcap\limits_{\alpha}G_{\alpha}^t$ for every $t>0$.
\end{Proposition}

Indeed, the order closure $\overline{I\mathfrak O}$ consists of functions from $\mathfrak F$ which are limits of at least one net from $I\mathfrak O$. For $g   \in
\overline{I\mathfrak O}$ and a net $\{\hat G_{\alpha}\}$ in $\mathfrak O$ such that $I\hat G_{\alpha} \overset o\to g$ one can take $G_{\alpha} :=\bigcup\limits_{\beta>\alpha}\hat G_{\beta}$. This is a decreasing net and, besides that, $G_{\alpha}^t =\bigcup\limits_{\beta>\alpha}\hat G_{\beta}^t$. The remaining assertions of the Proposition can be also easily checked. We add that there can be no {\it growing} net which converges to $g$.

\begin{Lemma}\label{lem 2 kernel}
Let $g\in\overline{I\mathfrak O}$ and $\{G_{\alpha}\}$ be a decreasing net in $\mathfrak O$ such that $IG_{\alpha} \overset o\to g$. Then
$\bigcap\limits_{t>0}g(t)   =
\bigcap\limits_{t>0}\overline{g(t)}   =
\bigcap\limits_{\alpha}\overline{G_{\alpha}}$.
\end{Lemma}

\begin{proof}
Pick an arbitrary $\alpha$. For every $t>0$ one has $g(t)={\rm int}\bigcap\limits_{\beta}G_{\beta}^t   \subseteq   {\rm int}G_{\alpha}^t   =   G_{\alpha}^t$. This means that $\overline{g(t)} \subseteq   \overline{G_{\alpha}^t}   \subseteq\{x:d(x,G_{\alpha})\leqslant t\}   \subseteq   G_{\alpha}^{2t}$. Therefore $\bigcap\limits_{t>0}\overline{g(t)}   \subseteq   \bigcap\limits_{t>0}G_{\alpha}^{2t}=\overline{G_{\alpha}}$. By arbitrariness of $\alpha$ one has
\begin{equation}\label{eq 2 1}
\bigcap\limits_{t>0}\overline{g(t)}   \subseteq   \bigcap\limits_{\alpha}\overline{G_{\alpha}}.
\end{equation}

On the other hand,
$$
\bigcap\limits_{\alpha}\overline{G_{\alpha}} \subseteq   \left(\bigcap\limits_{\alpha}\overline{G_{\alpha}}\right)^t   \subseteq   {\rm int}\bigcap\limits_{\alpha}(\overline{G_{\alpha}})^t   =   {\rm int}\bigcap\limits_{\alpha}G_{\alpha}^t   =   g(t)
$$
for every $t>0$, and hence $\bigcap\limits_{\alpha}\overline{G_{\alpha}}\subseteq\bigcap\limits_{t>0}g(t)$. Together with \eqref{eq 2 1} this gives the assertion of the lemma.
\end{proof}

\noindent$\bullet$\,\,\,For every function $g\in\overline{I\mathfrak O}$ we consider the set
\begin{equation*}\label{eq 2 g-dot}
\dot g:=\bigcap\limits_{t>0}g(t)\,\subset\,\Omega
\end{equation*}
(which is always closed by Lemma \ref{lem 2 kernel}) and call it \emph{the nucleus} of the element $g$. For every $g   \in   \overline{I\mathfrak O}$ one has
\begin{equation}\label{eq 2 lower bound}
(\dot g)^t \,  \subseteq \,  g(t), \qquad t>0.
\end{equation}
Indeed, using Lemma \ref{lem 2 kernel}, one has:
\begin{equation*}
(\dot g)^t   =   \left(\bigcap\limits_{\alpha}\overline{G_{\alpha}}\right)^t   \subseteq   {\rm int}\bigcap\limits_{\alpha}(\overline{G_{\alpha}})^t   =   {\rm int}\bigcap\limits_{\alpha}G_{\alpha}^t   =   g(t).
\end{equation*}
The relation \eqref{eq 2 lower bound} becomes trivial, if the nucleus is empty. However there is a simple condition which provides non-emptiness of the nucleus.

\begin{Condition}\label{cond compactness}
For every $x\in \Omega$ and $r>0$ the closed ball $B_r[x]$ is compact.
\end{Condition}

\begin{Lemma}\label{lem 2 nonzero}
Under Condition \ref{cond compactness}, $g\in\overline{I\mathfrak O}$ and $g\not=0_{\mathfrak F}$ imply $\dot g\neq\emptyset$.
\end{Lemma}

\begin{proof}
We make use of the following simple statement:
\begin{equation}\label{eq obvious}
A\cap B^t=\emptyset\quad\Leftrightarrow\quad A^t\cap B=\emptyset.
\end{equation}
By Proposition \ref{prop 2 convergence 2 monotonic + convergence}, for $g$ there exists a decreasing net $\{G_{\alpha}\}$ such that $IG_{\alpha} \overset o\to g$ and $\dot g   = \bigcap\limits_{\alpha}\overline{G_{\alpha}}$. Let $\dot g   =   \emptyset$. For every $x   \in   \Omega$ and $t>0$ one has $B_t[x]\cap\left(\bigcap\limits_{\alpha}\overline{G_{\alpha}}\right)   =   \emptyset$. It follows from compactness of the closed ball that there exists a finite set of indices $\alpha_1,\alpha_2,...,\alpha_n$ such that $B_t[x]\cap\left(\bigcap\limits_{i=1}^n\overline{G_{\alpha_i}}\right)   =   \emptyset$. Since the set of indices is directed, there exists $\gamma$, an upper bound of  the set $\{\alpha_1,\alpha_2,...,\alpha_n\}$, and then $B_t[x]\cap\overline{G_{\gamma}}   =   \emptyset$. This means that $\{x\}^t   \cap   G_{\gamma}   =   \emptyset$ which by the statement above is equivalent to $\{x\}   \cap   G_{\gamma}^t = \emptyset$ or to $x \notin G_{\gamma}^t$. Therefore $x \notin \bigcap\limits_{\alpha}G_{\alpha}^t$ for every $x$ and $t$. Hence $g(t)   =   {\rm int}\bigcap\limits_{\alpha}G_{\alpha}^t = \emptyset$ for every $t>0$, which leads to a contradiction.
\end{proof}

As we see, under Condition \ref{cond compactness} the relation \eqref{eq 2 lower bound} is a meaningful {\it estimate from below} for functions from $\overline{I\mathfrak O}$.
\smallskip

\noindent$\bullet$\,\,\,In what follows we will need one more condition on the metric space $(\Omega,d)$.

\begin{Condition}\label{cond distance}
For every $x,y\in\Omega$ and $r,s>0$, $B_r(x)\cap B_s(y)   =   \emptyset$ implies $d(x,y)   \geqslant   r+s$.
\end{Condition}

Let us remark that the class of metric spaces which satisfy Conditions \ref{cond compactness} and \ref{cond distance} contains {\em complete locally compact spaces with inner metric} \cite{Burago-Burago-Ivanov} (including Riemannian manifolds).

\begin{Proposition}\label{prop 2 semigroup}
Under Condition \ref{cond distance}, for every $A\subseteq \Omega$ one has $(A^r)^s   =   A^{r+s}$.
\end{Proposition}

\begin{Corollary*}
Under Condition \ref{cond distance}, for every $x\in\Omega$ and $r>0$ one has $\overline{B_r(x)}=B_r[x]$.
\end{Corollary*}

The following lemma has technical character.

\begin{Lemma}\label{lem technical}
Let Conditions \ref{cond compactness} and \ref{cond distance} hold. Then for every decreasing net $\{A_{\alpha}\}$ of sets from the space $\Omega$ one has
\begin{equation*}\label{eq 2 general 2-side bound}
\left(\bigcap\limits_{\alpha}A_{\alpha}\right)^t   \subseteq   \bigcap\limits_{\alpha}A_{\alpha}^t
\subseteq   \overline{\left(\bigcap\limits_{\alpha}\overline{A_{\alpha}}\right)^t},
\quad t>0.
\end{equation*}
\end{Lemma}

\begin{proof}
The first inclusion is always true, let us prove the second. Let $x\notin\overline{\left(\bigcap\limits_{\alpha}\overline{A_{\alpha}}\right)^t}$. Then, by separation axioms, there exists $r>0$ such that $B_r(x)   \cap   \overline{\left(\bigcap\limits_{\alpha}\overline{A_{\alpha}}\right)^t}   =   \emptyset$. Hence $B_r(x)   \cap \left(\bigcap\limits_{\alpha}\overline{A_{\alpha}}\right)^t   =   \emptyset$, and using \eqref{eq obvious} and Proposition \ref{prop 2 semigroup}, one has $(B_r(x))^t   \cap   (\bigcap\limits_{\alpha}\overline{A_{\alpha}})   =   B_{r+t}(x)   \cap   \left(\bigcap\limits_{\alpha}\overline{A_{\alpha}}\right) =   \emptyset$, and therefore $B_t[x]   \cap   \left(\bigcap\limits_{\alpha}\overline{A_{\alpha}}\right)   =   \emptyset$. The closed ball is compact, hence there exists a finite set of indices $\{\alpha_1, \alpha_2,...,\alpha_n\}$ such that $B_t[x]   \cap   \left(\bigcap\limits_{i=1}^n\overline{A_{\alpha_i}}\right)   =   \emptyset$. The net $\{\overline{A_{\alpha}}\}$ decreases; hence for $\gamma   =   \sup\{\alpha_1,\alpha_2,...,\alpha_n\}$ one has $B_t[x]   \cap   \overline{A_{\gamma}}=\emptyset$. Consequently, $B_t(x)   \cap   A_{\gamma}   = \emptyset$, which due to \eqref{eq obvious} is equivalent to $x \notin A_{\gamma}^t$, thus $x \notin\bigcap\limits_{\alpha}A_{\alpha}^t$. Therefore $\bigcap\limits_{\alpha}A_{\alpha}^t   \subseteq \overline{\left(\bigcap\limits_{\alpha}\overline{A_{\alpha}}\right)^t}$, which completes the proof.
\end{proof}

\noindent$\bullet$\,\,\,Now we can obtain an {\it estimate from above} for an element of $\overline{I\mathfrak O}$ in terms of its nucleus.

\begin{Lemma}\label{lem 2 upper bound}
Let Conditions \ref{cond compactness} and \ref{cond distance} hold. Then for every function $g   \in   \overline{I\mathfrak O}$ one has
\begin{equation}\label{eq 2 upper bound}
g(t) \,  \subseteq   \,{\rm int\,}\overline{{\dot g}^t}, \quad t   >   0.
\end{equation}
\end{Lemma}

\begin{proof}
If $g\equiv\emptyset$, then the assertion is obvious. Let $g\not\equiv\emptyset$. Then by Lemma \ref{lem technical} one has $\bigcap\limits_{\alpha}G_{\alpha}^t \subseteq\overline{\left(\bigcap\limits_{\alpha}\overline{G_{\alpha}}\right)^t}=\overline{\dot g^t}$, therefore $g(t)={\rm int}\bigcap\limits_{\alpha}G_{\alpha}^t\subseteq{\rm int\,}\overline{\dot g^t}$.
\end{proof}

\noindent$\bullet$\,\,\, Let as show that the set $\overline{I\mathfrak O}$ contains sufficient amount of functions which nuclei are single points. Define $b_*[x], b^*[x]\in {\mathfrak F}$ by the equalities
\begin{equation}\label{eq def b*[x]}
(b_*[x])(t)=B_t(x), \quad (b^*[x])(t)={\rm int}\,B_t[x],\qquad t>0,
\end{equation}
note that $b_*[x]\leqslant b^*[x]$.

\begin{Lemma}\label{lem 3 segment non-empty}
Let Condition \ref{cond distance} hold. Then for every $x\in\Omega$ one has $b^*[x]\in\overline{I\mathfrak O}$, and the nucleus of the function $b^*[x]$ consists of the single point $x$.
\end{Lemma}

\begin{proof}
Let us prove that the net $\{IB_{\varepsilon}(x)\}_{\varepsilon>0}$ converges in $\mathfrak F$ to $b^*[x]$ as $\varepsilon\rightarrow+0$. Since this net is decreasing, by Proposition \ref{prop 1 monotonic} one needs to check that ${\rm int\,}\bigcap\limits_{\varepsilon>0} (B_{\varepsilon}(x))^t   =   {\rm int\,}   B_t[x]$ for every $t>0$. From Proposition \ref{prop 2 semigroup} one has $(B_{\varepsilon}(x))^t=B_{\varepsilon+t}(x)   = (B_{t}(x))^{\varepsilon}$. Furthermore, $\bigcap\limits_{\varepsilon>0} (B_{\varepsilon}(x))^t =\bigcap\limits_{\varepsilon>0} (B_{t}(x))^{\varepsilon} =\overline{B_{t}(x)}=B_t[x]$, by Corollary from Proposition \ref{prop 2 semigroup}. From this it follows that
$$
{\rm int}   \bigcap\limits_{\varepsilon>0} (B_{\varepsilon}(x))^t={\rm int\,}   B_t[x],  \qquad t>0.
$$
The nucleus of the function $b^*[x]$, which is $\bigcap\limits_{t>0}{\rm int\,}   B_t[x]$, obviously contains the point $x$ and cannot contain any other points, which means it coincides with $\{x\}$.
\end{proof}

\section{The wave model}
$\bullet$\,\,\,Consider an equivalence relation on functions from the set $\overline{I\mathfrak O}$ defined as follows:
\begin{equation*}
f_1\sim f_2 \quad \Leftrightarrow \quad \bigcap_{t>0}f_1(t)=\bigcap_{t>0}f_2(t)\quad\Leftrightarrow\quad\dot f_1=\dot f_2.
\end{equation*}
Let $\langle f\rangle$ be the equivalence class of the function $f$. For the functions defined in \eqref{eq def b*[x]} one has $\dot b_*[x]=\dot b^*[x]=\{x\}$, and so $b_*[x]\sim b^*[x]$ and $\langle b_*[x]\rangle=\langle b^*[x]\rangle$. All class representatives have the same nucleus, which allows us to define the latter as {\it the nucleus of the equivalence class} denoted in what follows by $\langle...\rangle^{^\centerdot}$. Thus on has $\langle b_*[x]\rangle^{^\centerdot}=\langle b^*[x]\rangle^{^\centerdot}=\{x\}$.

On the set of classes $\overline{I\mathfrak O}\slash_{\sim}$ define a partial order in the following way:
\begin{equation*}
\langle f\rangle\leqslant \langle g\rangle\quad\Leftrightarrow\quad\bigcap_{t>0}f(t)\subseteq\bigcap_{t>0}g(t)\quad\Leftrightarrow\quad\dot f\subseteq\dot g.
\end{equation*}

\noindent$\bullet$\,\,\,Let ${\mathscr P}$ be a partially ordered set with the least element $0$. An element $a\in\mathscr P$ is called {\it an atom} ($a\in {\rm At}\mathscr P$), if $0<p\leqslant a$ implies $p=a$\,\,\,\cite{Birkhoff}. Note that the lattice of open sets $\mathfrak O$ can contain no atoms, e. g., in the case $\Omega=\mathbb R^n$.

The partially ordered set $\overline{I\mathfrak O}\slash_{\sim}$ has the least element $0_{\mathfrak F}$, and therefore one can speak of its atoms
\begin{equation*}\label{eq 3 atoms}
\widetilde{\Omega} \,  :=  \, {\rm At\,}   (\overline{I\mathfrak
O}\slash_{\sim}).
\end{equation*}
The set $\widetilde{\Omega}$ is the main object of the present paper.

\begin{Lemma}
Let Conditions \ref{cond compactness} and \ref{cond distance} hold. Then $\widetilde\Omega=\{\langle b_*[x]\rangle\,|\,\, x\in\Omega\}$.
\end{Lemma}

\begin{proof}
Let us prove that for every $x\in\Omega$ the class $\langle b_*[x]\rangle$ is an atom. Since its nucleus consists of only one point $x$, every class which is strictly less than $\langle b_*[x]\rangle$ should have empty nucleus, but then this class should be the least. Consequently $\langle b_*[x]\rangle$ is an atom. To prove the converse, let $\langle a\rangle$ be an atom. Then for every $x\in \langle a\rangle^{^\centerdot}$ the inequality $\langle b_*[x]\rangle\leqslant\langle a\rangle$ holds. By the definition of atom this means that $\langle a\rangle=\langle b_*[x]\rangle$ and $\langle a\rangle^{^\centerdot} = \{x\}$.
\end{proof}

Let the functions $a,b\in\overline{I\mathfrak O}$ be representatives of the atoms $\langle a\rangle$ and $\langle b\rangle$, respectively. Let us call the function
\begin{equation*}\label{eq 3 tau}
\tau(\langle a\rangle,\langle b\rangle)   :=   2{\,\rm inf\,}\{t\   |\   a(t) \cap   b(t) \neq \emptyset\}
\end{equation*}
the {\it wave distance} between these atoms. The choice of this name is motivated by applications where the points corresponding to the atoms initiate waves which propagate in $\Omega$ with the unit speed. At the moment $t=\frac{\tau(\langle a\rangle,\langle b\rangle)}{2}$ these waves start to overlap \cite{B JOT}. Correctness of this definition (finiteness for every pair of atoms, independence of the choice of representatives) follows from the next lemma.

\begin{Lemma}\label{lem 3 distance}
Let the Conditions \ref{cond compactness} and \ref{cond distance} hold. Then for every $x,y\in\Omega$ one has $\tau(\langle b_*[x]\rangle, \langle b_*[y]\rangle)= d(x,y)$.
\end{Lemma}

\begin{proof}
Since for every representatives $a\in\langle b_*[x]\rangle$ and $b\in\langle b_*[y]\rangle$ the inclusions $B_t(x) \subseteq   a(t)\subseteq   {\rm int}B_t[x]$ and $B_t(y) \subseteq   b(t) \subseteq {\rm int}B_t[y]$ hold, for $t<\frac{d(x,y)}2$ one has $a(t)\cap b(t)\subseteq{\rm int}B_t[x]\cap{\rm int}B_t[y] =\emptyset$. For $t>\frac{d(x,y)}2$ it is true that $a(t)\cap b(t)\supseteq   B_t(x) \cap   B_t(y)   \neq \emptyset$, because if $B_t(x)   \cap B_t(y)   =   \emptyset$, then by Condition \ref{cond distance} $d(x,y)\geqslant 2t$. Therefore ${\,\rm inf\,}\{t\   |\ a(t)   \cap   b(t)   \neq \emptyset\} =\frac{d(x,y)}2$, and this proves the lemma.
\end{proof}

Let us call the space $(\widetilde\Omega,\tau)$  {\it the wave model} of the original space $(\Omega,d)$. One can summarize the preceding considerations to formulate the main result of the paper.

\begin{Theorem}
Under Conditions \ref{cond compactness} and \ref{cond distance} the wave model $(\widetilde\Omega,\tau)$ is isometric to the original $(\Omega,d)$.
\end{Theorem}

The isometry is realized by the bijection $\Omega\ni x\mapsto \langle b_*[x]\rangle\in \tilde\Omega$.

\section*{Examples and comments}
Passing to equivalence classes is important. One can prove that in the set $\overline{I\mathfrak O}$ atoms correspond to points of the space $\Omega$. At the same time several atoms can correspond to one point and be incomparable to each other, while belonging to the segment $\left[\,[b_*[x],b^*[x]\,\right]:=\{g\in{\mathfrak F}\,|\,\,b_*[x]\leqslant g\leqslant b^*[x]\}$. We demonstrate this fact by examples.
\smallskip

\noindent$\bullet$\,\,\,In $\Omega=\mathbb R^n$ a unique atom of the set $\overline{I\mathfrak O}$ corresponds to every $x\in\mathbb R^n$: $a_x(t)=B_t(x),\,\,t>0$, $\dot a_x=\{x\}$, so that the wave model can be constructed without factorization.
\smallskip

\noindent$\bullet$ If $\Omega=[0,1],\,\,d(x,y)=|x-y|$, then for every $x\in(0,1)$ the segment $\left[\,b_*[x],b^*[x]\,\right]$ consists of four functions (see Fig. \ref{fig 1}); at the same time $B_t(x)$ does not belong to the set $\overline{I\mathfrak O}$, and two of these functions ($a^{(1)}$ and $a^{(2)}$ on the figure) are atoms of $\overline{I\mathfrak O}$. For $x\in(0,\frac12]$ the atoms are
$$
a^{(1)}(t)   :=
\begin{cases}
B_t(x), & 0<t<1-x,\ t\neq x,\\
[0,2x), & t=x,\\
[0,1), & t=1-x,\\
\Omega, & t>1-x
\end{cases}
$$
and
$$
a^{(2)}(t)   :=
\begin{cases}
B_t(x), & 0<t<1-x,\ t\neq x,\\
(0,2x), & t=x,\\
\Omega, & t\geqslant1-x.
\end{cases}
$$
One can take $G^{(1)}_{\varepsilon}=(x-\varepsilon,x)$ and $G^{(2)}_{\varepsilon}=(x,x+\varepsilon)$ as the initial approximating sets for these atoms: for $\varepsilon\to+0$ one has $(G^{(i)}_{\varepsilon})^t\to a^{(i)}(t)$ for $i=1,2$. On the figure the points marked with small circles are excluded. After factorization all the four functions become identical.

\begin{figure}[h]
\begin{center}
\includegraphics[width=\textwidth]{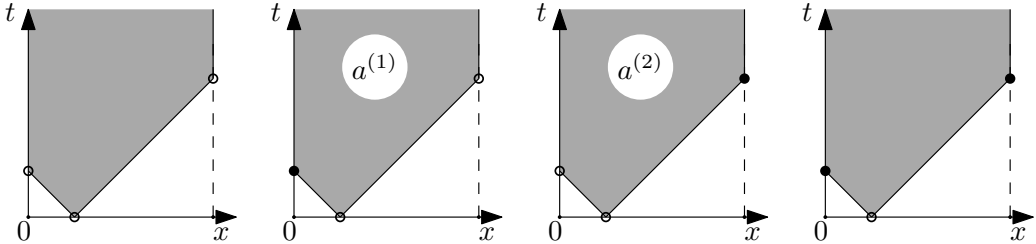}
\caption{The elements of the segment $\left[\,[b_*[x],b^*[x]\,\right]$} \label{fig 1}
\end{center}
\end{figure}

\noindent$\bullet$\,\,\,Situation in which more than one atom of the set $\overline{I\mathfrak O}$ corresponds to one point of $\Omega$ is possible for Riemmanian manifolds, if for large $t$ the boundaries of the balls $B_t(x)$ have self-intersections. Besides that, the picture gets more complicated, if $\Omega$ has a boundary. Nevertheless, in this case Conditions \ref{cond compactness} and \ref{cond distance} are satisfied and the wave model is isometric to the manifold. This fact plays the key role for the problem of reconstructing a manifold from the inverse spectral and dynamical data: cf. \cite{B JOT}.
\smallskip

\noindent$\bullet$\,\,\,In \cite{B JOT}, p. 303, a topology (non-Hausdorff) was introduced on ${\rm At\,}\overline{I\mathfrak O}$, which separates atoms of this set.
\smallskip

\noindent$\bullet$\,\,\,It would be interesting to find out to which extent the wave model remains meaningful, if the Conditions \ref{cond compactness} and \ref{cond distance} are weakened. For example, in the case of the space with the discrete metric
$$
d(x,y)=
\begin{cases}
1, & x\not=y,\\
0, & x=y,
\end{cases}
$$
Condition \ref{cond distance} is not satisfied and we have $\tau(x,y)=2d(x,y)$. The wave model is isometric to the original ``up to a homothety''.


\begin{thebibliography}{9}

\bibitem{B JOT}
M. I. Belishev.
\newblock {A unitary invariant of a semi-bounded operator in reconstruction
of manifolds.}
\newblock{\em Journal of Operator Theory}, Volume 69 (2013),
Issue 2, 299--326.


\bibitem{B_UMN}
M. I. Belishev.
\newblock{Boundary control and tomography of Riemannian
manifolds (the BC-method).}
\newblock {\em Uspekhi Mat. Nauk}, 72 (2017), no. 4, 3--66.
Russian Math. Surveys; DOI 10.1070/RM9768.


\bibitem{BSim_1}
M. I. Belishev, S. A. Simonov.
\newblock {Wave model of the Sturm-Liouville operator on the
half-line.} \newblock{\em St. Petersburg Math. J.},
29 (2018), no. 2, 227--248;
arXiv: 1703.00176v1.


\bibitem{Birkhoff}
G. Birkhoff. Lattice Theory.
\newblock{\em AMS, Providence, Rhode Island}, 1967.


\bibitem{Vulih}
B. Z. Vulikh.
Introduction to the Theory of Partially Ordered Spaces.
{\em Wolters-Noordhoff}, 1967.


\bibitem{Burago-Burago-Ivanov}
D. Burago, Yu. Burago, S. Ivanov.
\newblock{A Course in Metric Geometry.}
\newblock{\em AMS, Providence, Rhode Island}, 2001.


\bibitem{Simonov 2017}
S. A. Simonov.
\newblock {Wave model of the regular Sturm Liouville operator}.
\newblock{\em 2017 Days on Diffraction (DD)},
St. Petersburg, 2017, pp. 300--303; DOI 10.1109/DD.2017.8168043;
arXiv 1801.02011.


\end{thebibliography}
\end{document}